\documentclass[12pt]{article}
\usepackage{amscd}
\usepackage{amsfonts}
\usepackage{amsmath}
\usepackage{amssymb}
\usepackage{amsthm}
\usepackage{bbm}
\usepackage{CJK}
\usepackage{fancyhdr}
\usepackage{graphicx}
\usepackage{indentfirst}
\usepackage{latexsym,bm}
\usepackage{mathrsfs}
\usepackage{xypic}
\usepackage[square, comma, sort&compress, numbers]{natbib}
\allowdisplaybreaks[4]
\newtheorem{theorem}{Theorem}[section]

\newtheorem{definition}[theorem]{Definition}
\newtheorem{proposition}[theorem]{Proposition}
\newtheorem{example}[theorem]{Example}

\newtheorem{corollary}[theorem]{Corollary}
\newtheorem{remark}[theorem]{Remark}

\usepackage[top=1in,bottom=1in,left=1.25in,right=1.25in]{geometry}
\def\<{\langle}
\def\>{\rangle}
\def\a{\alpha}
\def\b{\beta}

\def\d{\delta}

\def\g{\gamma}

\def\o{\otimes}

\def\r{\rho}

\def\tr{\triangleright}
\def\tl{\triangleleft}

\date{}
\begin{document}
\renewcommand{\baselinestretch}{1.2}
\renewcommand{\arraystretch}{1.0}
\title{\bf The construction of braided $T$-category via Yetter-Drinfeld-Long bimodules}
 \date{}
\author {{\bf Daowei Lu$^{1,2}$, Yan Ning$^{2}$, Dingguo Wang$^1$\footnote{Corresponding author: dgwang@qfnu.edu.cn} }\\
{\small $^1$School of Mathematical Sciences, Qufu Normal University}\\
{\small Qufu, Shandong 273165, P. R. China}\\
{\small $^2$Department of Mathematics, Jining University}\\
{\small Qufu, Shandong 273155, P. R. China}
}
\maketitle
\begin{center}
\begin{minipage}{15.cm}
\noindent{\bf Abstract.}
Let $H_1$ and $H_2$ be Hopf algebras which are not necessarily finite dimensional and $\a,\b\in Aut_{Hopf}(H_1),\g,\d\in Aut_{Hopf}(H_2)$. In this paper, we introduce a category $_{H_1}\mathcal{LR}_{H_2}(\a,\b,\g,\d)$, generalizing Yetter-Drinfeld-Long bimodules and construct a braided $T$-category $\mathcal{LR}(H_1,H_2)$ containing all the categories $_{H_1}\mathcal{LR}_{H_2}(\a,\b,\g,\d)$ as components. We also prove that if $(\a,\b,\g,\d)$ admits a quadruple in involution, then $_{H_1}\mathcal{LR}_{H_2}(\a,\b,\g,\d)$ is isomorphic to the usual category $_{H_1}\mathcal{LR}_{H_2}$ of Yetter-Drinfeld-Long bimodules.
\\

\\

\end{minipage}
\end{center}

\section*{Introduction}

Braided $T$-category, introduced by Turaev in \cite{Tur00} (see also \cite{Tur94,Tur08,Tur10}), is of intense interest since it could give rise to homotopy quantum field theories. A. Kirillov \cite{Kir} found that braided $T$-categories also provide a suitable mathematical tool to describe the orbifold models which arise in the study of conformal field theories,
and play a key role in the construction of Hennings-type invariants of flat group-bundles over complements of link in the 3-sphere (see \cite{Vir}). As such they are interesting to different research communities in mathematical physics (for example, see \cite{FY1989, Kir}).

In the theory of Hopf algebra, there are mainly two methods for the construction of braided $T$-category. One is to construct Turaev group-coalgebra, the representative category of which is a braided $T$-category. For details the readers can refer to \cite{Fang13,Zhang,Wang}. The other was proposed by F. Panaite and M. Staic in their paper \cite{Pan07}, which is based on some kind of generalized Yetter-Drinfeld module of Hopf algebra with bijective antipode. Since then their method was followed and generalized to other Hopf algebra structure, such as in \cite{Liu10} to weak Hopf algebra, in \cite{Yang} to multiplier Hopf algebra, in \cite{You} to monoidal Hom-Hopf algebra, in \cite{Liu19} to weak monoidal Hom-Hopf algebra and so on. And recently in \cite{Lu2} the first author constructed a braided $T$-category by a method dual to that of \cite{Pan07}.

In \cite{Pan10} F. Panaite and F. Van Oystaeyen introduced a braided monoidal category, the category $\mathcal{LR}(H)$ of Yetter-Drinfeld-Long bimodules of a Hopf algebra $H$, which is more generalized than Yetter-Drinfeld category $^H_H\mathcal{YD}$. As we all know, a Hopf algebra in $^H_H\mathcal{YD}$ corresponds to the classic Radford biproduct, while more generally the Hopf algebra in $\mathcal{LR}(H)$ corresponds to L-R smash biproduct also defined in \cite{Pan10}. In \cite{Lu1} the first author pointed out that, for $H$ being finite dimensional, $\mathcal{LR}(H)$ is actually a Yetter-Drinfeld category.

Motivated by these results, in this paper, we will construct a more generalized braided $T$-category than the one in \cite{Pan07} via the category $_{H_1}\mathcal{LR}_{H_2}(\a,\b,\g,\d)$  of generalized Yetter-Drinfeld-Long bimodules with the Hopf algebras $H_1,H_2$( not being necessarily finite dimensional) and $\a,\b\in Aut(H_1),\g,\d\in Aut(H_2)$.

This paper is organized as follows. In section 1, we will recall the notions of crossed group category and braided $T$-category. In section 2, we will give the construction of braided $T$-category $\mathcal{LR}(H_1,H_2)$, and prove that the subcategory $\mathcal{LR}(H_1,H_2)_{fd}$ consisting of finite dimensional objects is rigid. In section 3, we will apply our construction to the case of group algebras. In section 4, we will give the isomorphism of category $_{H_1}\mathcal{LR}_{H_2}(\a,\b,\g,\d)\simeq\! _{H_1}\mathcal{LR}_{H_2}$ via a quadruple $(f_1,g_1,f_2,g_2)$ in involution corresponding to $(\a,\b,\g,\d)$.

\section{Preliminaries}
\def\theequation{1.\arabic{equation}}
\setcounter{equation} {0}

Throughout this paper, let $k$ be a fixed field, and all vector spaces and tensor product are over $k$.

Let $G$ be a group with the unit 1. Recall from \cite{Tur00} that a crossed $T$-category $\mathcal{C}$ over $G$ is given by the following data:

$\bullet $ $\mathcal{C}$ is a strict monoidal category.

$\bullet $ A family of subcategories $\{\mathcal{C_\a}\}_{\a\in G}$ such that $\mathcal{C}$ is a disjonit union of this family and that $U\o V\in\mathcal{C}_{\a\b}$ for any $\a,\b\in G$, $U\in\mathcal{C}_\a$ and $V\in\mathcal{C}_\b$.

$\bullet $ A group homomorphism $\varphi:G\rightarrow Aut(\mathcal{C}),\b\mapsto\varphi_{_\b}$, the $conjugation$, where $Aut(\mathcal{C})$ is the group of the invertible strict tensor functors from $\mathcal{C}$ to itself, such that $\varphi_{_\b}(\mathcal{C_\a})=\mathcal{C}_{\b\a\b^{-1}}$ for any $\a,\b\in G$.

We will use the Turaev's left index notation: Given $\b\in G$ and an object $V\in\mathcal{C_\a}$, the functor $\varphi_{_\b}$ will be denoted by $^{\b}(\cdot)$ or $^{V}(\cdot)$ and $^{\b^{-1}}(\cdot)$ will be denoted by $^{\overline{V}}(\cdot)$. Since $^{V}(\cdot)$ is a functor, for any object $U\in\mathcal{C}$ and any composition of morphism $g\circ f$ in $\mathcal{C}$, we obtain $^{V}id_U=id_{^{V}U}$ and $^{V}(g\circ f)=\! ^{V}g\circ\! ^{V}f$. Since the conjugation $\varphi:\pi\rightarrow aut(\mathcal{C})$ is a group homomorphism, for any $V,W\in\mathcal{C}$, we have $^{V\o W}(\cdot)=\! ^{V}(^{W}(\cdot))$ and $^{1}(\cdot)=\! ^{V}(^{\overline{V}}(\cdot))=\! ^{\overline{V}}(^{V}(\cdot))=id_{\mathcal{C}}$. Since for any $V\in\mathcal{C}$, the functor $^{V}(\cdot)$ is strict, we have $^{V}(f\o g)=\! ^{V}f\o\! ^{V}g$ for any morphism $f$ and $g$ in $\mathcal{C}$, and $^{V}(1)=1.$

A braided $T$-category is a crossed $T$-category $\mathcal{C}$ endowed with a braiding, i.e., a family of isomorphisms
$$c=\{c_{_{U,V}}:U\o V\rightarrow~ ^{V}U\o V\}_{U,V\in\mathcal{C}}
$$
obeying the following conditions:

$\bullet $ For any morphism $f\in Hom_{\mathcal{C}_\a}(U,U')$ and $g\in Hom_{\mathcal{C}_\b}(V,V')$, we have
$$(^{\a}g\o f)\circ c_{_{U,V}}=c_{_{U',V'}}\circ(f\o g),
$$

$\bullet $ For all $U,V,W\in\mathcal{C}$, we have
\begin{align}
c_{_{U\o V,W}}&=(c_{_{U,^{V}W}}\o id_V)(id_U\o c_{_{V,W}}),\label{1a}\\
c_{_{U,V\o W}}&=(id_{^{U}V}\o c_{_{U,W}})(c_{_{U,V}}\o id_{W}).\label{1b}
\end{align}

$\bullet $ For any $U,V\in\mathcal{C}$ and $\a\in G$, $\varphi_{_\a}(c_{_{U,V}})=c_{_{^{\a}U,^{\a}V}}$.

\section{The construction of $\mathcal{LR}(H_1,H_2)$}
\def\theequation{2.\arabic{equation}}
\setcounter{equation} {0}

The aim of this section is to construct the braided $T$-category $\mathcal{LR}(H_1,H_2)$. First of all, we need to give the following definition.

\begin{definition}
Let $H_1,H_2$ be two Hopf algebras, $\a,\b\in Aut_{Hopf}(H_1),\g,\d\in Aut_{Hopf}(H_2)$. The category $_{H_1}\mathcal{LR}_{H_2}(\a,\b,\g,\d)$ is defined as follows. The objects are vector spaces\\ $M$ endowed with $H_1$-$H_2$-bimodule and $H_1$-$H_2$-bicomodule
structures (denoted by $h'\o m \mapsto h'\tr m, m\o h''\mapsto m\tl h'', \r_1(m)=m_{(0)}\o m_{(1)},
\r_2(m)= m_{[-1]}\o m_{[0]}$, for all $h'\in H_1,h''\in H_2$, $m\in M$), such that $M$ is a left-right $(\a,\b)$-Yetter-Drinfeld module under $(\tr,\r_1)$, a right-left $(\g,\d)$-Yetter-Drinfeld module under$(\tl,\r_2)$, a
left and right Long module, i.e.
\begin{align}
&\a(h'_1)m_{(-1)}\o h'_2\tr m_{(0)}=(h'_1\tr m)_{(-1)}\b(h'_2)\o (h'_1\tr m)_{(0)},\label{2a}\\
&(h'\tr m)_{[0]}\o(h'\tr m)_{[1]}=h'\tr m_{[0]}\o m_{[1]} ,\label{2b}\\
&m_{[0]}\tl h''_1 \o m_{[1]}\d(h''_2)=(m\tl h''_2)_{[0]}\o \g(h''_1)(m\tl h''_2)_{[1]},\label{2c}\\
&(m\tl h'')_{(-1)}\o(m\tl h'')_{(0)}=m_{(-1)}\o m_{(0)}\tl h''.\label{2d}
\end{align}
The morphisms in $_{H_1}\mathcal{LR}_{H_2}(\a,\b,\g,\d)$ are $H_1$-$H_2$-bilinear and $H_1$-$H_2$-bicolinear. The object in $_{H_1}\mathcal{LR}_{H_2}(\a,\b,\g,\d)$ is called $(\a,\b,\g,\d)$-Yetter-Drinfeld-Long bimodule.
\end{definition}

\begin{remark}
\begin{itemize}
  \item [(1)]In the above definition, the condition (\ref{2a}) is equivalent to
\begin{equation}\label{2e}
(h'\tr m)_{(-1)}\o(h'\tr m)_{(0)}=\a(h_1)m_{(-1)}\b(S(h_3))\o h'_2\tr m_{(0)}.
\end{equation}
The condition (\ref{2c}) is equivalent to
\begin{equation}\label{2f}
(m\tl h'')_{[0]}\o(m\tl h'')_{[1]}=m_{[0]}\tl h''_2\o \g(S(h''_1))m_{[1]}\d(h_3).
\end{equation}
  \item [(2)] Denote $_{H_1}\mathcal{LR}_{H_2}=\! _{H_1}\mathcal{LR}_{H_2}(id,id,id,id)$. Obviously the category $\mathcal{LR}(H)$ introduced in \cite{Pan10} is a special case of $_{H_1}\mathcal{LR}_{H_2}$ when $H_1=H_2$.
\end{itemize}

\end{remark}

\begin{example}
Let $\a,\b\in Aut(H_1),\g,\d\in Aut(H_2)$, and assume that there exist algebra maps $f_1:H_1\rightarrow k,f_2:H_2\rightarrow k$ and group-like elements $g_1\in H_1,g_2\in H_2$ such that for all $h'\in H_1,h''\in H_2$,
\begin{align*}
&\a(h')=g_1f_1(h'_1)\b(h'_2)f_1(S(h'_3))g^{-1}_1,\\
&\d(h'')=g^{-1}_2f_2(S(h''_1))\g(h''_2)f_2(h''_3)g_2.
\end{align*}
For any vector space $V$, define
\begin{align*}
&h'\tr v=\varepsilon_1(h')v,\ v_{(-1)}\o v_{(0)}=g_1\o v,\\
&v\tl h''=\varepsilon_2(h'')v,\ v_{[0]}\o v_{[1]}=v\o g_2,
\end{align*}
for all $v\in V$. Then one can check that $V\in\! _{H_1}\mathcal{LR}_{H_2}(\a,\b,\g,\d)$. We will call $(f_1,g_1,f_2,g_2)$ a quadruple in involution corresponding to $(\a,\b,\g,\d)$.
\end{example}

\begin{proposition}
Let $M\in\! _{H_1}\mathcal{LR}_{H_2}(\a_1,\b_1,\g_1,\d_1),N\in\! _{H_1}\mathcal{LR}_{H_2}(\a_2,\b_2,\g_2,\d_2)$, then $M\o N\in\! _{H_1}\mathcal{LR}_{H_2}(\a_2\a_1,\a_2\b_1\a^{-1}_2\b_2,\g_1\g_2,\d_2\g^{-1}_2\d_1\g_2)$ with the following structure
\begin{align*}
&h'\tr(m\o n)=h'_1\tr m\o h'_2\tr n,\\
&(m\o n)_{(-1)}\o(m\o n)_{(0)}=\a_2(m_{(-1)})\a_2\b_1\a^{-1}_2(n_{(-1)})\o m_{(0)}\o n_{(0)},\\
&(m\o n)\tl h''=m\tl \g_2(h''_1)\o n\tl\g^{-1}_2\d_1\g_2(h''_2),\\
&(m\o n)_{[0]}\o(m\o n)_{[1]}=m_{[0]}\o n_{[0]}\o m_{[1]}n_{[1]}.
\end{align*}
\end{proposition}

\begin{proof}
Clearly $M\o N$ is an $H_1$-$H_2$-bimodule and $H_1$-$H_2$-bicomodule. For all $h'\in H_1,m\in M$,
\begin{align*}
&[h'_1\tr(m\o n)]_{(-1)}\a_2\b_1\a^{-1}_2\b_2(h'_2)\o[h'_1\tr(m\o n)]_{(0)}\\
&=[h'_1\tr m\o h'_2\tr n]_{(-1)}\a_2\b_1\a^{-1}_2\b_2(h'_3)\o[h'_1\tr m\o h'_2\tr n]_{(0)}\\
&=\a_2((h'_1\tr m)_{(-1)})\a_2\b_1\a^{-1}_2((h'_2\tr n)_{(-1)})\a_2\b_1\a^{-1}_2\b_2(h'_3)\o(h'_1\tr m)_{(0)}\o(h'_2\tr n)_{(0)}\\
&=\a_2((h'_1\tr m)_{(-1)})\a_2\b_1(h'_2)\a_2\b_1\a^{-1}_2(n_{(-1)})\o(h'_1\tr m)_{(0)}\o h'_3\tr n_{(0)}\\
&=\a_2\a_1(h'_1)\a_2(m_{(-1)})\a_2\b_1\a^{-1}_2(n_{(-1)})\o h'_2\tr m_{(0)}\o h'_3\tr n_{(0)}\\
&=\a_2\a_1(h'_1)(m\o n)_{(-1)}\o h'_2\tr (m\o  n)_{(0)},
\end{align*}
which proves (\ref{2a}) and the verification of (\ref{2b}) and (\ref{2c}) is straightforward.
\begin{align*}
&[(m\o n)\tl h'']_{(-1)}\o[(m\o n)\tl h'']_{(0)}\\
&=[m\tl \g_2(h''_1)\o n\tl\g^{-1}_2\d_1\g_2(h''_2)]_{(-1)}\o[m\tl \g_2(h''_1)\o n\tl\g^{-1}_2\d_1\g_2(h''_2)]_{(0)}\\
&=\a_2[(m\tl \g_2(h''_1))_{(-1)}]\a_2\b_1\a^{-1}_2[(n\tl\g^{-1}_2\d_1\g_2(h''_2))_{(-1)}]\\
&\quad \o(m\tl \g_2(h''_1))_{(0)}\o (n\tl\g^{-1}_2\d_1\g_2(h''_2))_{(0)}\\
&=\a_2(m_{(-1)})\a_2\b_1\a^{-1}_2(n_{(-1)})\o m_{(0)}\tl \g_2(h''_1)\o n_{(0)}\tl\g^{-1}_2\d_1\g_2(h''_2)\\
&=(m\o n)_{(-1)}\o(m\o n)_{(0)}\tl h'',
\end{align*}
which proves (\ref{2d}). The proof is completed.
\end{proof}

Denote $G_i=Aut_{Hopf}(H_i)\times Aut_{Hopf}(H_i),i=1,2$.
\begin{itemize}
  \item Define multiplication of $G_1$ by
$$(\a_1,\b_1)*(\a_2,\b_2)=(\a_2\a_1,\a_2\b_1\a^{-1}_2\b_2),$$
Then $G_1$ is a group with unit $(id,id)$ and $(\a,\b)^{-1}=(\a^{-1},\a^{-1}\b^{-1}\a)$.
  \item Define multiplication of $G_2$ by
$$(\g_1,\d_1)*(\g_2,\d_2)=(\g_1\g_2,\d_2\g^{-1}_2\d_1\d_2),$$
Then $G_2$ is a group with unit $(id,id)$ and $(\g,\d)^{-1}=(\g^{-1},\g\d^{-1}\g^{-1})$.
\end{itemize}
Let $G=G_1\oplus G_2$ be the direct sum of $G_1$ and $G_2$.

With the above notations, Proposition 2.2 could be rewritten as

{\it If $M\in\! _{H_1}\mathcal{LR}_{H_2}(\a_1,\b_1,\g_1,\d_1),N\in\! _{H_1}\mathcal{LR}_{H_2}(\a_2,\b_2,\g_2,\d_2)$, then
$$M\o N\in\! _{H_1}\mathcal{LR}_{H_2}((\a_1,\b_1,\g_1,\d_1)*(\a_2,\b_2,\g_2,\d_2)).$$}
Note that if $M\in\! _{H_1}\mathcal{LR}_{H_2}(\a_1,\b_1,\g_1,\d_1),N\in\! _{H_1}\mathcal{LR}_{H_2}(\a_2,\b_2,\g_2,\d_2),$\\ $P\in\! _{H_1}\mathcal{LR}_{H_2}(\a_3,\b_3,\g_3,\d_3)$, then $(M\o N)\o P=M\o(N\o P)\in\! _{H_1}\mathcal{LR}_{H_2}((\a_1,\b_1,\g_1,\d_1)*(\a_2,\b_2,\g_2,\d_2)*(\a_3,\b_3,\g_3,\d_3))$.

\begin{proposition}
Let $N\in\! _{H_1}\mathcal{LR}_{H_2}(\a_2,\b_2,\g_2,\d_2)$ and $(\a_1,\b_1,\g_1,\d_1)\in G$. Define $^{(\a_1,\b_1,\g_1,\d_1)}N=N$ as a vector space with structures
\begin{align*}
&h'\rightharpoonup n=\b^{-1}_1\a_1(h')\tr n,\\
&n_{\<-1\>}\o n_{\<0\>}=\a^{-1}_1\a_2\b_1\a^{-1}_2(n_{(-1)})\o n_{(0)},\\
&n\leftharpoonup h''=n\tl \g^{-1}_2\d_1\g_2 \g^{-1}_1(h''),\\
&n_{\{0\}}\o n_{\{1\}}=n_{[0]}\o \g_1\d^{-1}_1(n_{[1]}).
\end{align*}
for all $h'\in H_1,h''\in H_2,n\in N$. Then
\begin{align*}
^{(\a_1,\b_1,\g_1,\d_1)}N&\in\! _{H_1}\mathcal{LR}_{H_2}(\a^{-1}_1\a_2\a_1,\a^{-1}_1\a_2\b_1\a^{-1}_2\b_2\b^{-1}_1\a_1,\g_1\g_2\g^{-1}_1,\g_1\d^{-1}_1\d_2\g^{-1}_2\d_1\g_2\g^{-1}_1)\\
&=\! _{H_1}\mathcal{LR}_{H_2}((\a_1,\b_1,\g_1,\d_1)*(\a_2,\b_2,\g_2,\d_2)*(\a_1,\b_1,\g_1,\d_1)^{-1}).
\end{align*}
\end{proposition}

\begin{proof}
Obviously $^{(\a_1,\b_1,\g_1,\d_1)}N$ is an $H_1$-$H_2$-bimodule and $H_1$-$H_2$-bicomodule. For all $h'\in H_1,n\in n$,
\begin{align*}
&(h'\rightharpoonup n)_{\<-1\>}\o(h'\rightharpoonup n)_{\<0\>}\\
&=\a^{-1}_1\a_2\b_1\a^{-1}_2[(\b^{-1}_1\a_1(h')\tr n)_{(-1)}]\o(\b^{-1}_1\a_1(h')\tr n)_{(0)}\\
&=\a^{-1}_1\a_2\b_1\a^{-1}_2[\a_2\b^{-1}_1\a_1(h'_1)n_{(-1)}\b_2\b^{-1}_1\a_1(S(h'_3))]\o\b^{-1}_1\a_1(h'_2)\tr n_{(0)}\\
&=\a^{-1}_1\a_2\a_1(h'_1)\a^{-1}_1\a_2\b_1\a^{-1}_2(n_{(-1)})\a^{-1}_1\a_2\b_1\a^{-1}_2\b_2\b^{-1}_1\a_1(S(h'_3))\o\b^{-1}_1\a_1(h'_2)\tr n_{(0)}\\
&=\a^{-1}_1\a_2\a_1(h'_1)n_{\<-1\>}\a^{-1}_1\a_2\b_1\a^{-1}_2\b_2\b^{-1}_1\a_1(S(h'_3))\o h'_2\rightharpoonup n_{\<0\>},
\end{align*}
which proves (\ref{2e}). And
\begin{align*}
&(h'\rightharpoonup n)_{\{0\}}\o(h'\rightharpoonup n)_{\{1\}}\\
&=(\b^{-1}_1\a_1(h')\tr n)_{[0]}\o\g_1\d^{-1}_1((\b^{-1}_1\a_1(h')\tr n)_{[1]})\\
&=\b^{-1}_1\a_1(h')\tr n_{[0]}\o\g_1\d^{-1}_1(n_{[1]})\\
&=h'\rightharpoonup n_{\{0\}}\o n_{\{1\}},
\end{align*}
which proves (\ref{2b}). For all $h''\in H_2$,
\begin{align*}
&(n\leftharpoonup h'')_{\{0\}}\o(n\leftharpoonup h'')_{\{1\}}\\
&=(n\tl \g^{-1}_2\d_1\g_2 \g^{-1}_1(h''))_{[0]}\o\g_1\d^{-1}_1((n\tl \g^{-1}_2\d_1\g_2 \g^{-1}_1(h''))_{[1]})\\
&=n_{[0]}\tl \g^{-1}_2\d_1\g_2 \g^{-1}_1(h''_2)\o\g_1\d^{-1}_1[\g_2S(\g^{-1}_2\d_1\g_2 \g^{-1}_1(h''_1))n_{[1]}\d_2\g^{-1}_2\d_1\g_2 \g^{-1}_1(h''_3)]\\
&=n_{[0]}\tl \g^{-1}_2\d_1\g_2 \g^{-1}_1(h''_2)\o\g_1\g_2\g^{-1}_1(S (h''_1))\g_1\d^{-1}_1(n_{[1]})\g_1\d^{-1}_1\d_2\g^{-1}_2\d_1\g_2 \g^{-1}_1(h''_3)\\
&=n_{\{0\}}\leftharpoonup h''_2\o\g_1\g_2\g^{-1}_1(S (h''_1))n_{\{1\}}\g_1\d^{-1}_1\d_2\g^{-1}_2\d_1\g_2 \g^{-1}_1(h''_3),
\end{align*}
which proves (\ref{2f}). And
\begin{align*}
&(n\leftharpoonup h'')_{\<-1\>}\o(n\leftharpoonup h'')_{\<0\>}\\
&=\a^{-1}_1\a_2\b_1\a^{-1}_2[(n\tl \g^{-1}_2\d_1\g_2 \g^{-1}_1(h''))_{(-1)}]\o(n\tl \g^{-1}_2\d_1\g_2 \g^{-1}_1(h''))_{(0)}\\
&=\a^{-1}_1\a_2\b_1\a^{-1}_2(n_{(-1)})\o n_{(0)}\tl \g^{-1}_2\d_1\g_2 \g^{-1}_1(h'')\\
&=n_{\<-1\>}\o n_{\<0\>}\leftharpoonup h'',
\end{align*}
which proves (\ref{2d}). Thus
$$^{(\a_1,\b_1,\g_1,\d_1)}N\in\! _{H_1}\mathcal{LR}_{H_2}((\a_1,\b_1,\g_1,\d_1)*(\a_2,\b_2,\g_2,\d_2)*(\a_1,\b_1,\g_1,\d_1)^{-1}).$$
\end{proof}

\begin{remark}
Let $M\in\! _{H_1}\mathcal{LR}_{H_2}(\a_1,\b_1,\g_1,\d_1),N\in\! _{H_1}\mathcal{LR}_{H_2}(\a_2,\b_2,\g_2,\d_2)$, and\\ $(\a_3,\b_3,\g_3,\d_3)\in G$. Then
$$^{(\a_1,\b_1,\g_1,\d_1)*(\a_3,\b_3,\g_3,\d_3)}N=^{(\a_1,\b_1,\g_1,\d_1)}\Big(^{(\a_3,\b_3,\g_3,\d_3)}N\Big)$$
is an object in $$_{H_1}\mathcal{LR}_{H_2}((\a_1,\b_1,\g_1,\d_1)*(\a_3,\b_3,\g_3,\d_3)*(\a_2,\b_2,\g_2,\d_2)*(\a_3,\b_3,\g_3,\d_3)^{-1}*(\a_1,\b_1,\g_1,\d_1)^{-1}).$$
And
$$^{(\a_3,\b_3,\g_3,\d_3)}(M\o N)=\!^{(\a_3,\b_3,\g_3,\d_3)}M\o\!^{(\a_3,\b_3,\g_3,\d_3)}N$$
is an object in $_{H_1}\mathcal{LR}_{H_2}((\a_3,\b_3,\g_3,\d_3)*(\a_1,\b_1,\g_1,\d_1)*(\a_2,\b_2,\g_2,\d_2)*(\a_3,\b_3,\g_3,\d_3)^{-1})$.
\end{remark}

\begin{proposition}
Let $M\in\! _{H_1}\mathcal{LR}_{H_2}(\a_1,\b_1,\g_1,\d_1),N\in\! _{H_1}\mathcal{LR}_{H_2}(\a_2,\b_2,\g_2,\d_2)$. Denote $^MN=\! ^{(\a_1,\b_1,\g_1,\d_1)}N$. Define the map
$$c_{M,N}:M\o N\rightarrow\! ^MN\o M,\ \ c_{M,N}(m\o n)=\b^{-1}_1(m_{(-1)})\tr n_{[0]}\o m_{(0)}\tl \d^{-1}_1(n_{[1]}).$$
Then $c_{M,N}$ is $H_1$-$H_2$-bilinear and $H_1$-$H_2$-bicolinear with the inverse
$$c^{-1}_{M,N}:\! ^MN\o M\rightarrow M\o N,\ \ c^{-1}_{M,N}(n\o m)=m_{(0)}\tl \d^{-1}_1(S^{-1}_2(n_{[1]}))\o\b^{-1}_1(S^{-1}_1(m_{(-1)}))\tr n_{[0]}.$$
Moreover $c_{M,N}$ satisfies relations (\ref{1a}) and (\ref{1b})
and for $P\in\! _{H_1}\mathcal{LR}_{H_2}(\a_3,\b_3,\g_3,\d_3)$,
$$c_{_{^PM,^PN}}=c_{M,N}.$$
\end{proposition}

\begin{proof}
It is easy to verify that $c_{M,N}$ is left $H_1$-linear and right $H_2$-colinear.
For all $m\in M,n\in N$ and $h''\in H_2$,
\begin{align*}
&c_{M,N}((m\o n)\tl h'')\\
&=c_{M,N}(m\tl \g_2(h''_1)\o n\tl \g^{-1}_2\d_1\g_2(h''_2))\\
&=\b^{-1}_1(m\tl \g_2(h''_1))_{(-1)})\tr(n\tl \g^{-1}_2\d_1\g_2(h''_2))_{[0]}\o (m\tl \g_2(h''_1))_{(0)}\tl \d^{-1}_1(n\tl \g^{-1}_2\d_1\g_2(h''_2))_{[1]})\\
&=\b^{-1}_1(m_{(-1)})\tr n_{[0]}\tl \g^{-1}_2\d_1\g_2(h''_1)\o m_{(0)}\tl \d^{-1}_1(n_{[1]}) \d^{-1}_1\d_2\g^{-1}_2\d_1\g_2(h''_2),
\end{align*}
and
\begin{align*}
&c_{M,N}(m\o n)\tl h''\\
&=(\b^{-1}_1(m_{(-1)})\tr n_{[0]}\o m_{(0)}\tl \d^{-1}_1(n_{[1]}))\tl h''\\
&=(\b^{-1}_1(m_{(-1)})\tr n_{[0]})\leftharpoonup \g_1(h''_1)\o m_{(0)}\tl \d^{-1}_1(n_{[1]})\g^{-1}_1\g_1\d^{-1}_1\d_2\g^{-1}_2\d_1\g_2\g^{-1}_1\g_1(h''_2)\\
&=\b^{-1}_1(m_{(-1)})\tr n_{[0]}\tl \g^{-1}_2\d_1\g_2(h''_1)\o m_{(0)}\tl \d^{-1}_1(n_{[1]})\d^{-1}_1\d_2\g^{-1}_2\d_1\g_2(h''_2).
\end{align*}
Thus $c_{M,N}$ is right $H_2$-linear. Since
\begin{align*}
&c_{M,N}(m\o n)_{(-1)}\o c_{M,N}(m\o n)_{(0)}\\
&=(\b^{-1}_1(m_{(-1)})\tr n_{[0]}\o m_{(0)}\tl \d^{-1}_1(n_{[1]}))_{(-1)}\o(\b^{-1}_1(m_{(-1)})\tr n_{[0]}\o m_{(0)}\tl \d^{-1}_1(n_{[1]}))_{(0)}\\
&=\a_1((\b^{-1}_1(m_{(-1)})\tr n_{[0]})_{\<-1\>})\a_2\b_1\a^{-1}_2\b_2\b^{-1}_1((m_{(0)}\tl \d^{-1}_1(n_{[1]}))_{(-1)})\\
&\quad\o(\b^{-1}_1(m_{(-1)})\tr n_{[0]})_{\<0\>}\o(m_{(0)}\tl \d^{-1}_1(n_{[1]}))_{(0)}\\
&=\a_2(m_{(-1)1})\a_2\b_1\a^{-1}_2(n_{[0](-1)})\a_2\b_1\a^{-1}_2\b_2\b^{-1}_1(S(m_{(-1)3}))\a_2\b_1\a^{-1}_2\b_2\b^{-1}_1(m_{(-1)4})\\
&\quad\o\b^{-1}_1(m_{(-1)2})\tr n_{[0](0)}\o m_{(0)}\tl \d^{-1}_1(n_{[1]})\\
&=\a_2(m_{(-1)})\a_2\b_1\a^{-1}_2(n_{(-1)})\o \b^{-1}_1(m_{(0)(-1)})\tr n_{(0)[0]}\o m_{(0)(0)}\tl \d^{-1}_1(n_{(0)[1]})\\
&=\a_2(m_{(-1)})\a_2\b_1\a^{-1}_2(n_{(-1)})\o c_{M,N}(m_{(0)}\o n_{(0)})\\
&=(m\o n)_{(-1)}\o c_{M,N}((m\o n)_{(0)}),
\end{align*}
we obtain that $c_{M,N}$ is left $H_1$-colinear. For all $M\in\! _{H_1}\mathcal{LR}_{H_2}(\a_1,\b_1,\g_1,\d_1),N\in\! _{H_1}\mathcal{LR}_{H_2}(\a_2,\b_2,\g_2,\d_2), P\in\! _{H_1}\mathcal{LR}_{H_2}(\a_3,\b_3,\g_3,\d_3)$ and all $m\in M,n\in N,p\in P$, on one hand,
\begin{align*}
&c_{M\o N,P}(m\o n\o p)\\
&=(\a_2\b_1\a^{-1}_2\b_2)^{-1}((m\o n)_{(-1)})\tr p_{[0]}\o(m\o n)_{(0)}\tl\d_2\g^{-1}_2\d_1\g_2(p_{[1]})\\
&=(\a_2\b_1\a^{-1}_2\b_2)^{-1}(\a_2(m_{(-1)})\a_2\b_1\a^{-1}_2(n_{(-1)}))\tr p_{[0]}\o (m_{(0)}\o n_{(0)})\tl\d_2\g^{-1}_2\d_1\g_2(p_{[1]})\\
&=\b^{-1}_2\a_2\b^{-1}_1(m_{(-1)})\b^{-1}_2(n_{(-1)})\tr p_{[0]}\o m_{(0)}\tl \d^{-1}_1\g_2\d^{-1}_2(p_{[1]1})\o n_{(0)}\tl\d^{-1}_2(p_{[1]2}).
\end{align*}
On the other hand,
\begin{align*}
&(c_{M,\! ^NP}\o id)(id\o c_{N,P})(m\o n\o p)\\
&=(c_{M,\! ^NP}\o id)(m\o \b^{-1}_2(n_{(-1)})\tr p_{[0]}\o n_{(0)}\tl \d^{-1}_2(p_{[1]}))\\
&=\b^{-1}_1(m_{(-1)})\rightharpoonup(\b^{-1}_2(n_{(-1)})\tr p_{[0]})_{\{0\}}\o m_{(0)}\tl\d^{-1}_1((\b^{-1}_2(n_{(-1)})\tr p_{[0]})_{\{1\}})\o n_{(0)}\tl \d^{-1}_2(p_{[1]})\\
&=\b^{-1}_2\a_2\b^{-1}_1(m_{(-1)})\b^{-1}_2(n_{(-1)})\tr p_{[0]}\o m_{(0)}\tl\d^{-1}_1\g_2\d^{-1}_2(p_{[0][1]})\o n_{(0)}\tl \d^{-1}_2(p_{[1]})\\
&=\b^{-1}_2\a_2\b^{-1}_1(m_{(-1)})\b^{-1}_2(n_{(-1)})\tr p_{[0]}\o m_{(0)}\tl\d^{-1}_1\g_2\d^{-1}_2(p_{[1]1})\o n_{(0)}\tl \d^{-1}_2(p_{[1]2}),
\end{align*}
which implies the identity (\ref{1a}). Similarly we could verify the identity (\ref{1b}).
%
%
For the last statement,
\begin{align*}
&c_{_{^PM,^PN}}(m\o n)\\
&=\a^{-1}_3\b_3\b^{-1}_1\a_1\b^{-1}_3\a^{-1}_1\a_3(m_{\<-1\>})\rightharpoonup n_{\{0\}}\o m_{\<0\>}\leftharpoonup\g_3\g_1\d^{-1}_3\g_1\d^{-1}_1\d_3\g^{-1}_3(n_{\{1\}})\\
&=\a^{-1}_3\b_3\b^{-1}_1(m_{(-1)})\rightharpoonup n_{[0]}\o m_{(0)}\leftharpoonup\g_3\g_1\d^{-1}_3\g_1\d^{-1}_1(n_{[1]})\\
&=\b^{-1}_1(m_{(-1)})\tr n_{[0]}\o m_{(0)}\tl\d^{-1}_1(n_{[1]})\\
&=c_{M,N}(m\o n).
\end{align*}
The proof is completed.
\end{proof}

Let $\mathcal{LR}(H_1, H_2)$ be the disjoint union of all $ _{H_1}\mathcal{LR}_{H_2}(\a,\b,\g,\d)$, with $(\a,\b,\g,\d)\in G=G_1\oplus G_2$. Endowed with the tensor product in Proposition 2.2, $\mathcal{LR}(H_1, H_2)$ becomes a strict monoidal category
with a unit $k$ as an object in $ _{H_1}\mathcal{LR}_{H_2}(id,id,id,id)$.

The group homomorphism $\varphi:G\rightarrow Aut(\mathcal{LR}(H_1, H_2)),\ (\a,\b,\g,\d)\mapsto\varphi_{(\a,\b,\g,\d)}$, is given on components as
\begin{align*}
\varphi_{(\a_1,\b_1,\g_1,\d_1)}:\! _{H_1}\mathcal{LR}_{H_2}&(\a_2,\b_2,\g_2,\d_2)\rightarrow\\
&\rightarrow \! _{H_1}\mathcal{LR}_{H_2}((\a_1,\b_1,\g_1,\d_1)*(\a_2,\b_2,\g_2,\d_2)*(\a_1,\b_1,\g_1,\d_1)^{-1}),\\
&\varphi_{(\a_1,\b_1,\g_1,\d_1)}(N)=\! ^{(\a_1,\b_1,\g_1,\d_1)}N,
\end{align*}
and the functor acts on morphisms as identity. The braiding in $\mathcal{LR}(H_1, H_2)$ is given by the the family $c=\{c_{M,N}\}$. By the above results, we have the following theorem.

\begin{theorem}
$\mathcal{LR}(H_1, H_2)$ is a braided $T$-category over $G$.
\end{theorem}

Suppose $H$ to be a finite dimensional Hopf algebra. In \cite{Lu1} it is shown that $\mathcal{LR}(H)\simeq\! _{H\o H^*}^{H\o H^*}\mathcal{YD}$ as braided monoidal categories. By a similar process, if $H_1$ and $H_2$ are finite dimensional, we have $_{H_1}\mathcal{LR}_{H_2}\simeq\! _{H_1\o H^*_2}^{H_1\o H^*_2}\mathcal{YD}$ as braided monoidal categories.

Just as our construction, let $\mathcal{YD}(H_1\o H^*_2)$ be the braided $T$-category over the group $Aut(H_1\o H^*_2)$. Then we have the following result.
\begin{proposition}
Let $H_1$ and $H_2$ be finite dimensional Hopf algebras. Then $\mathcal{LR}(H_1, H_2)\simeq\mathcal{YD}(H_1\o H^*_2)$ as braided $T$-categories.
\end{proposition}

\begin{proof}
Let $H$ be a finite dimensional Hopf algebra, then
$$\varphi:Aut(H)\rightarrow Aut(H^*),\ \ \alpha\mapsto\varphi_\a:H^*\rightarrow H^*,\ \varphi_\a(f)(h)=f(\a(h)),$$
is an anti-isomorphism of group. By a similar procedure as in \cite{Lu1}, we could obtain this isomorphism.
\end{proof}

Now we consider the existence of left and right dualities.
\begin{proposition}
Let $M$ be a finite dimensional object in $\! _{H_1}\mathcal{LR}_{H_2}(\a,\b,\g,\d)$, then $M^*=Hom(M,k)$ becomes an object in $\! _{H_1}\mathcal{LR}_{H_2}(\a^{-1},\a^{-1}\b^{-1}\a,\g^{-1},\g\d^{-1}\g^{-1})$ with the following structures: for all $h'\in H_1,h''\in H_2,f\in M^*,m\in M$,
\begin{align*}
&(h'\tr f)(m)=f(S_1(h')\tr m),\\
&f_{(-1)}\o f_{(0)}(m)=\a^{-1}\b^{-1}S^{-1}_1(m_{(-1)})\o f(m_{(0)}),\\
&(f\tl h'')(m)=f(m\tl \d^{-1}\g^{-1}S^{-1}_2(h'')),\\
&f_{[0]}(m)\o f_{[1]}=f(m_{[0]})\o S_2(m_{[1]}).
\end{align*}
Moreover $M^*$ is the left duality of $M$.
\end{proposition}

\begin{proof}
For all $h'\in H_1,h''\in H_2,f\in M^*,m\in M$,
\begin{align*}
&\a^{-1}(h'_1)f_{(-1)}\o(h'_2\tr f_{(0)})(m)\\
&=\a^{-1}(h'_1)f_{(-1)}\o  f_{(0)}(S_1(h')\tr m)\\
&=\a^{-1}(h'_1)\a^{-1}\b^{-1}S^{-1}_1((S_1(h')\tr m)_{(-1)})\o f((S_1(h')\tr m)_{(0)})\\
&=\a^{-1}\b^{-1}S^{-1}_1(m_{(-1)})\a^{-1}\b^{-1}\a(h'_2)\o f(S_1(h'_1)\tr m_{(0)})\\
&=\a^{-1}\b^{-1}S^{-1}_1(m_{(-1)})\a^{-1}\b^{-1}\a(h'_2)\o (h'_1\tr f)(m_{(0)})\\
&=(h'_1\tr f)_{(-1)}\a^{-1}\b^{-1}\a(h'_2)\o (h'_1\tr f)_{(0)}(m),
\end{align*}
which proves
$$\a^{-1}(h'_1)f_{(-1)}\o(h'_2\tr f_{(0)})=(h'_1\tr f)_{(-1)}\a^{-1}\b^{-1}\a(h'_2)\o (h'_1\tr f)_{(0)}.$$
And
\begin{align*}
&(f_{[0]}\tl h''_1)(m)\o f_{[1]}\g\d\g^{-1}(h''_2)\\
&=f_{[0]}(m\tl\d^{-1}\g^{-1}S^{-1}_2(h''_1))\o f_{[1]}\g\d\g^{-1}(h''_2)\\
&=f((m\tl\d^{-1}\g^{-1}S^{-1}_2(h''_1))_{[0]})\o S_2((m\tl\d^{-1}\g^{-1}S^{-1}_2(h''_1))_{[1]})\g\d\g^{-1}(h''_2)\\
&=f(m_{[0]}\tl\d^{-1}\g^{-1}S^{-1}_2(h''_2)))\o S_2(\g\d^{-1}\g^{-1}(h''_3)m_{[1]}\g^{-1}S^{-1}_2(h''_1))\g\d^{-1}\g^{-1}(h''_4)\\
&=f(m_{[0]}\tl\d^{-1}\g^{-1}S^{-1}_2(h''_2)))\o \g^{-1}(h''_1)S_2(m_{[1]})\\
&=(f\tl h''_2)(m_{[0]})\o\g^{-1}(h''_1)S_2(m_{[1]})\\
&=(f\tl h''_2)_{[0]}(m)\o\g^{-1}(h''_1)(f\tl h''_2)_{[1]},
\end{align*}
which proves
$$(f_{[0]}\tl h''_1)\o f_{[1]}\g\d\g^{-1}(h''_2)=(f\tl h''_2)_{[0]}\o\g^{-1}(h''_1)(f\tl h''_2)_{[1]}.$$
Also
\begin{align*}
&(h'\tr f)_{[0]}(m)\o (h'\tr f)_{[1]}=(h'\tr f)(m_{[0]})\o S_2(m_{[1]})\\
&=f(S_1(h')\tr m_{[0]})\o S_2(m_{[1]})=f((S_1(h')\tr m)_{[0]})\o S_2((S_1(h')\tr m)_{[1]})\\
&=f_{[0]}(S_1(h')\tr m)\o f_{[1]}=(h'\tr f_{[0]})(m)\o f_{[1]},
\end{align*}
and
\begin{align*}
&(f\tl h'')_{(-1)}\o(f\tl h'')_{(0)}(m)=\a^{-1}\b^{-1}S^{-1}_1(m_{(-1)})\o (f\tl h'')(m_{(0)})\\
&=\a^{-1}\b^{-1}S^{-1}_1(m_{(-1)})\o f(m_{(0)}\tl \d^{-1}\g^{-1}S^{-1}_2(h''))\\
&=\a^{-1}\b^{-1}S^{-1}_1((m\tl \d^{-1}\g^{-1}S^{-1}_2(h''))_{(-1)})\o f((m\tl \d^{-1}\g^{-1}S^{-1}_2(h''))_{(0)})\\
&=f_{(-1)}\o f_{(0)}(m\tl \d^{-1}\g^{-1}S^{-1}_2(h''))=f_{(-1)}\o (f_{(0)}\tl h'')(m).
\end{align*}
Therefore $M^*$ is an object in $\! _{H_1}\mathcal{LR}_{H_2}(\a^{-1},\a^{-1}\b^{-1}\a,\g^{-1},\g\d^{-1}\g^{-1})$.
Define the value map $ev$ and covalue map $coev$, respectively as follows:
\begin{align*}
&ev:M^*\o M\rightarrow k,\ \ f\o m\mapsto f(m),\\
&coev:k\rightarrow M\o M^*,\ \ 1\mapsto \sum_i m_i\o m^i,
\end{align*}
where $\{m_i\}$ and $\{m^i\}$ are dual bases in $M$ and $M^*$. We prove that $ev$ and $coev$ are $H_1$-$H_2$-bilinear. Indeed
\begin{align*}
&ev(h'\tr(f\o m)\tl h'')=ev(h'_1\tr f\tl \g(h''_1)\o h'_2\tr m\tl \d^{-1}(h''_2))\\
&=(h'_1\tr f\tl \g(h''_1))(h'_2\tr m\tl \d^{-1}(h''_2))\\
&=\varepsilon(h')\varepsilon(h'')f(m),
\end{align*}
and
\begin{align*}
&h'\tr(\sum_i m_i\o m^i)\tl h''\\
&=\sum_i h'_1\tr m_i\tl \g^{-1}(h''_1)\o h'_2\tr m^i\tl \g\d\g^{-1}(h''_2).
\end{align*}
For all $m\in M$,
\begin{align*}
&\sum_i h'_1\tr m_i\tl \g^{-1}(h''_1)\o (h'_2\tr m^i\tl \g\d\g^{-1}(h''_2))(m)\\
&=\sum_i h'_1\tr m_i\tl \g^{-1}(h''_1)\o m^i(S_1(h'_2)\tr m\tl \g^{-1}S^{-1}_2(h''_2))\\
&=\varepsilon(h')\varepsilon(h'')m,
\end{align*}
which implies $coev(h'\tr 1\tl h'')=h'\tr coev(1)\tl h''$.

Next we prove that $ev$ and $coev$ are $H_1$-$H_2$-bicolinear.
\begin{align*}
&(f\o m)_{(-1)}\o ev((f\o m)_{(0)})\\
&=\a(f_{(-1)})\b^{-1}(m_{(-1)})\o f_{(0)}(m_{(0)})\\
&=\b^{-1}S^{-1}_1(m_{(0)(-1)})\b^{-1}(m_{(-1)})\o f(m_{(0)(0)})\\
&=f(m)1\o 1=(ev(f\o m))_{-1}\o (ev(f\o m))_{0},
\end{align*}
\begin{align*}
&(\sum_i m_i\o m^i)_{(-1)}\o (\sum_i m_i\o m^i)_{(0)}(m)\\
&=\sum_i \a^{-1}(m_{i(-1)})\o \a^{-1}\b\a(m^i_{(-1)})\o m_{i(0)}\o m^i_{(0)}(m)\\
&=\sum_i \a^{-1}(m_{i(-1)})\a^{-1}S^{-1}_1(m_{(-1)})\o m_{i(0)}\o m^i(m_{(0)})\\
&=1\o m=(coev(1)\o 1)(m).
\end{align*}
It is straightforward to check that $coev$ is right $H_2$-colinear. Finally it is trivial to verify the identities
\begin{align*}
&(id_M\o ev)(coev\o id_M)=id_M,\\
&(ev\o id_{M^*})(id_{M^*}\o coev)=id_{M^*}.
\end{align*}
Therefore $M^*$ is the left duality of $M$. The proof is completed.
\end{proof}

Similarly we have the following result.
\begin{proposition}
Let $M$ be a finite dimensional object in $\! _{H_1}\mathcal{LR}_{H_2}(\a,\b,\g,\d)$, then $^*M=Hom(M,k)$ becomes an object in $\! _{H_1}\mathcal{LR}_{H_2}(\a^{-1},\a^{-1}\b^{-1}\a,\g^{-1},\g\d^{-1}\g^{-1})$ with the following structures: for all $h'\in H_1,h''\in H_2,f\in\! ^*M,m\in M$,
\begin{align*}
&(h'\tr f)(m)=f(S^{-1}_1(h')\tr m),\\
&f_{(-1)}\o f_{(0)}(m)=\a^{-1}\b^{-1}S_1(m_{(-1)})\o f(m_{(0)}),\\
&(f\tl h'')(m)=f(m\tl \d^{-1}\g^{-1}S_2(h'')),\\
&f_{[0]}(m)\o f_{[1]}=f(m_{[0]})\o S^{-1}_2(m_{[1]}).
\end{align*}
Moreover $^*M$ is the right duality of $M$.
\end{proposition}

Let $\mathcal{LR}(H_1, H_2)_{fd}$ be the subcategory of $\mathcal{LR}(H_1, H_2)$ consisting of all finite dimensional objects. By the above two results, we have
\begin{theorem}
$\mathcal{LR}(H_1, H_2)_{fd}$ is a rigid braid $T$-category over $G$.
\end{theorem}

\section{The case of group algebras}
\def\theequation{3.\arabic{equation}}
\setcounter{equation} {0}

The aim of this section is to describe $\mathcal{LR}(H_1, H_2)$ if $H_1$ and $H_2$ are the group algebras $k[G_1]$ and $k[G_2]$ of the groups $G_1$ and $G_2$, respectively.

Let $\a,\b\in Aut(G_1),\g,\d\in Aut(G_2)$. An $(\a,\b,\g,\d)$-Yetter-Drinfeld-Long bimodule $M$ over $k[G_1]$ and $k[G_2]$ is a $G_1$-$G_2$-bimodule with a decomposition
$$M=\bigoplus_{g\in G_1,h\in G_2}\! _g M_h,$$
such that for $g,g'\in G_1,h,h'\in G_2,$ we have $g'\tr\! _{g} M_h \tl h'\subseteq\! _{\a(g')g\b(g'^{-1})} M_{\g(h'^{-1})h\d(h')} $.

If $\a_1,\b_1,\a_2,\b_2\in Aut(G_1),\g_1,\d_1,\g_2,\d_2\in Aut(G_2),$
$$M=\bigoplus_{g\in G_1,h\in G_2}\! _g M_h\in\! _{H_1}\mathcal{LR}_{H_2}(\a_1,\b_1,\g_1,\d_1),$$
and
$$N=\bigoplus_{g\in G_1,h\in G_2}\! _g N_h\in\! _{H_1}\mathcal{LR}_{H_2}(\a_2,\b_2,\g_2,\d_2).$$
Then
$$M\o N\in\! _{H_1}\mathcal{LR}_{H_2}(\a_2\a_1,\a_2\b_1\a^{-1}_2\b_2,\g_1\g_2,\d_2\g^{-1}_2\d_1\g_2),$$
with the actions
\begin{align*}
&g\tr(m\o n)=g\tr m\o g\tr n,\\
&(m\o n)\tl h=m\tl \g_2(h)\o n\tl\g^{-1}_2\d_1\g_2(h),
\end{align*}
and decompositions
$$M\o N=\bigoplus_{c\in G_1,d\in G_2}\Big(\bigoplus_{\a^{-1}_2(g_1)\a_2\b^{-1}_1\a^{-1}_2(g_2)=c,\ h_1h_2=d}\! _{g_1}M_{h_1}\o\!   _{g_2} N_{h_2}\Big).$$
Let $^{(\a_1,\b_1,\g_1,\d_1)}N=N$ as vector space with actions
$$g\rightharpoonup n=\b^{-1}_1\a_1(g)\tr n,\ \ n\leftharpoonup h=n\tl \g^{-1}_2\d_1\g_2\g^{-1}_1(h),$$
for all $n\in N,g\in G_1,h\in G_2$, and decompositions
$$N=\bigoplus\limits_{g\in G_1,h\in G_2}\! _{\a_2\b^{-1}_1\a^{-1}_2\a_1(g)} N_{\d_1\g^{-1}_1(h)}.$$
The braiding $c_{M,N}:M\o N\rightarrow\! ^MN\o M$ acts on homogeneous elements $m\in\! _{g_1}M_{h_1},n\in\! _{g_2} N_{h_2}$ as $c_{M,N}(m\o n)=\b^{-1}_1(g_1)\tr n\o m\tl \d^{-1}_1(h_2)$. Hence it sends $_{g_1}M_{h_1}\o\! _{g_2} N_{h_2}$ to $_{\a_2\b^{-1}_1(g_1)g_2\b_2\b^{-1}_1(g^{-1}_1)}N_{h_2}\o\! _{g_1}M_{\g_1\d^{-1}_1(h^{-1}_2)h_1h_2}$.

Let $M=\bigoplus_{g\in G_1,h\in G_2}\! _g M_h\in\! _{H_1}\mathcal{LR}_{H_2}(\a,\b,\g,\d)$ be finite dimensional. Since the antipode $S=S^{-1}$ for $k[G_1]$ and $k[G_2]$, we have $M^*=\! ^*M$, and its structure could be described as follows: for all $g\in G_1,m\in M,f\in M^*,$
\begin{align*}
&(g\tr f)(m)=f(g^{-1}\tr m),\\
&(f\tl h)(m)=f(m\tl \d^{-1}\g^{-1}(h^{-1})),\\
&\hbox{The decomposition is}\ M^*=\bigoplus_{g\in G_1,h\in G_2}\! (_{\b\a(g^{-1})}M_{h^{-1}})^*
\end{align*}

\section{An isomorphism of categories  $_{H_1}\mathcal{LR}_{H_2}(\a,\b,\g,\d)\simeq\! _{H_1}\mathcal{LR}_{H_2}$}
\def\theequation{4.\arabic{equation}}
\setcounter{equation} {0}

\begin{theorem}
Let $\a,\b\in Aut(H_1),\g,\d\in Aut(H_2)$, and assume that there exists a quadruple $(f_1,g_1,f_2,g_2)$ in involution corresponding to $(\a,\b,\g,\d)$. Then the categories  $_{H_1}\mathcal{LR}_{H_2}(\a,\b,\g,\d)$ and $ _{H_1}\mathcal{LR}_{H_2}$ are isomorphic.

A pair of inverse functors $(F,G)$ is given as follows. If $M\in\! _{H_1}\mathcal{LR}_{H_2}(\a,\b,\g,\d)$, then $F(M)\in\! _{H_1}\mathcal{LR}_{H_2}$, where $F(M)=M$ as vector space, with structures
\begin{align*}
&h'\rightarrow m=f_1(\b^{-1}S(h_1))\b^{-1}(h_2)\tr m,\ m_{\<-1\>}\o m_{\<0\>}=g_1m_{(-1)}\o m_{(0)},\\
&m\leftarrow h''=m\tl \g^{-1}(h''_1)f_2(\g^{-1}S(h''_2)),\ m_{\{0\}}\o m_{\{1\}}=m_{[0]}\o m_{[1]}g^{-1}_2.
\end{align*}
If $N\in\! _{H_1}\mathcal{LR}_{H_2}$, then $G(N)\in\! _{H_1}\mathcal{LR}_{H_2}(\a,\b,\g,\d)$, where $G(N) = N$ as vector space,
with structures
\begin{align*}
&h'\tr n=f_1(h'_1)\b(h'_2)\rightarrow n,\ n_{(-1)}\o n_{(0)}=g_1n_{\<-1\>}\o n_{\<0\>},\\
&n\tl h''=n\leftarrow\g(h''_1)f_2(h''_2),\ n_{[0]}\o n_{[1]}=n_{\{0\}}\o n_{\{1\}}g_2.
\end{align*}
Both $F$ and $G$ act as identities on morphisms.
\end{theorem}

\begin{proof}
The proof can be obtained by direct computations.
\end{proof}

\begin{corollary}
For any $\a\in Aut(H_1),\g\in Aut(H_2),$ we have
$$_{H_1}\mathcal{LR}_{H_2}(\a,\a,\g,\g)\simeq\! _{H_1}\mathcal{LR}_{H_2}.$$
\end{corollary}

\begin{proof}
Since $(\varepsilon_1, 1,\varepsilon_2,1)$ is a quadruple in involution corresponding to $(\a,\a,\g,\g)$, the isomorphism can be obtained by Theorem 4.1.
\end{proof}

\section*{Acknowledgement}
This work was supported by the NSF of China (No. 11871301, 11901240) and the NSF of Shandong Province (No. ZR2018PA006).


\begin{thebibliography}{aa}

\bibitem{Fang13}  X. L. Fang, S. H. Wang. New Turaev braided group categories and group corings based on quasi-Hopf group coalgebras, Commun. Algebra 41 (2013), 4195--4226.

\bibitem{FY1989} P. J. Freyd, D. N. Yetter. Braided compact closed categories with applications to low-dimensional topology.
   Adv. in Math. 77(1989), 156--182.

\bibitem{Kir} A. Kirillov. On $G$-equivariant modular categories, arXiv:math/0401119.

\bibitem{Liu19}G. H. Liu, W. Wang, S. H. Wang, X. H. Zhang.  A braided $T$-category over weak monoidal Hom-Hopf algebras, to appear in J. Alg. Appl.

\bibitem{Liu10} L. Liu, S. H. Wang. Constructing new braided $T$-categories over weak Hopf algebras, Appl. Categor. Struct. 18(2010), 431--459.

\bibitem{Lu1} D. W. Lu, S. H. Wang. Yetter-Drinfeld-Long bimodules are modules, Czech. Math. J. 67(2017), 379--387.

\bibitem{Lu2} D. W. Lu, M. M. You. A New Approach to the Constructions of Braided T-Categories, Filomat 31(2017), 6561--6574.

\bibitem{Pan07} F. Panaite, M. Staic. Generalized (anti)Yetter-Drinfeld modules as components of a braided $T$-categories, Isr. J. Math. 158(2007), 349--365.

\bibitem{Pan10} F. Panaite, F.Van Oystaeyen. L-R-smash biproducts, double biproducts and a braided
category of Yetter-Drinfeld-Long bimodules, Rocky Mt. J. Math. 40 (2010), 2013--2024.


\bibitem{Tur94} V. Turaev. Quantum invariants of knots and 3-manifolds. In: De Gruyter Stud. Math. vol. 18.
De Gruyter, Berlin(1994).

\bibitem{Tur00} V. Turaev. Homotopy field theory in dimension 3 and crossed group-categories, Preprint
(2000), GT/0005291.

\bibitem{Tur08} V. Turaev. Crossed group-categories, Arab. J. Sci. Eng. Sec. C 33(2008), 483--503.

\bibitem{Tur10} V. Turaev. Homotopy quantum field theory, with appendices by M. M¨¹ger and A. Virelizier.
 In: Tracts in Math., vol. 10. European Mathematical Society, Helsinki (2010).

\bibitem{Vir} A. Virelizier. Involutory Hopf group-coalgebras and flat bundles over 3-manifolds, Fund. Math.
 188(2005), 241--270.

\bibitem{Wang} S. H. Wang. New Turaev braided group categories and group Schur-Weyl duality, Appl. Categor. Struct. 21(2013), 141--166.

\bibitem{Yang} T. Yang, S. H. Wang. Constructing new braided $T$-categories over regualar multiplier Hopf algebras, Commu. Algebra 39(2011), 3073--3089.

\bibitem{You} M. M. You, S. H. Wang.  Constructing new braided T-categories over monoidal Hom-Hopf algebras, J. Math. Phy. 55(2014), 111701.

\bibitem{Zhang} X. H. Zhang, S. H. Wang. New Turaev braided group categories and weak(co)quasi-Turaev group coalgebras, J. Math. Phy. 55(2014), 111702.

\end{thebibliography}
\end{document}